\documentclass[12pt]{amsart}
\usepackage{geometry}    
\usepackage{amsmath, amscd} 
\usepackage{amssymb} 
\usepackage{xcolor} 

\usepackage{pst-node}
\usepackage{tikz-cd} 
\usepackage{graphicx} 
\usepackage{comment}
\usepackage[parfill]{parskip}    
\usepackage{hyperref} 

\makeatletter
\newtheorem*{rep@theorem}{\rep@title}
\newcommand{\newreptheorem}[2]{%
	\newenvironment{rep#1}[1]{%
		\def\rep@title{#2 \ref{##1}}%
		\begin{rep@theorem}}%
		{\end{rep@theorem}}}
\makeatother

\newtheorem{theorem}{Theorem}
\newreptheorem{theorem}{Theorem}
\newtheorem{lemma}{Lemma}
\newtheorem{prop}{Proposition}

\newreptheorem{cor}{Corollary}
\newreptheorem{prop}{Proposition}
\theoremstyle{definition}
\newtheorem{remark}[theorem]{Remark}

\newtheorem{question}{Question}

\newcommand{\rr}{\mathbb{R}}

\newcommand{\zz}{\mathbb{Z}}
\newcommand{\qq}{\mathbb{Q}}

\newcommand{\ff}{\mathbb{F}}

\DeclareMathOperator{\leanop}{\mathsf{Lean}}
\def\lean{\ensuremath{\leanop}} 

\title{(In)dependence of the Axioms of $\Lambda$-trees}
\author{Raphael Appenzeller }
\address{Department of Mathematics, ETH Zurich, Switzerland}
\email{raphael.appenzeller@math.ethz.ch}

\date{\today}      

\begin{document}

\def\subjclassname{\textup{2020} Mathematics Subject Classification}
\expandafter\let\csname subjclassname@1991\endcsname=\subjclassname
\date{\today}

\begin{abstract}
A $\Lambda$-tree is a $\Lambda$-metric space satisfying three axioms (1), (2) and (3). We give a characterization of those ordered abelian groups $\Lambda$ for which axioms (1) and (2) imply axiom (3). As a special case, it follows that for the important class of ordered abelian groups $\Lambda$ that satisfy $\Lambda=2\Lambda$, (3) follows from (1) and (2). For some ordered abelian groups $\Lambda$, we show that axiom (2) is independent of the axioms (1) and (3) and ask whether this holds for all ordered abelian groups. Part of this work has been formalized in the proof assistant \lean.
\end{abstract}

\maketitle

\section{Introduction}

Let $(\Lambda,+)$ be an ordered abelian group. A set $X$ together with  a $\Lambda$-valued function
$
    d \colon X \times X \to \Lambda 
$
that is positive definite, symmetric and satisfies the triangle inequality is called a \emph{$\Lambda$-metric space}. A $\Lambda$-metric space $(X,d)$ is a \emph{$\Lambda$-tree} if it satisfies the following three axioms (definitions in Section \ref{sec:def}):
\begin{itemize}
    \item[(1)] $(X,d)$ is geodesic: any two points can be joined by a segment.
    \item[(2)] If two segments $s,s'\subseteq X$ intersect in a single point $s\cap s' = \{p\}$, which is an endpoint of $s$ and $s'$, then their union $s\cup s'$ is a segment.
    \item[(3)] If two segments $s,s'\subseteq X$ have an endpoint in common, then their intersection $s\cap s'$ is a segment. 
\end{itemize}

 It is known that for $\Lambda = \zz$ and $\Lambda=\rr$, axiom (3) follows from (1) and (2), see \cite[Lemma I.1.1]{ms84} or \cite[Lemma I.2.3]{ch01}. As $\Lambda$-trees generalize trees from graph theory ($\Lambda = \zz$) and real trees ($\Lambda = \rr$) the following question arises. For which ordered abelian groups $\Lambda$ is the addition of axiom (3) necessary?
In Theorem \ref{thm:main}, we give a complete characterization of the groups $\Lambda$, for which axiom (3) follows from (1) and (2), answering the question. 

\begin{theorem}\label{thm:main}
	Let $\Lambda$ be an ordered abelian group. The following are equivalent:
	\begin{itemize}
		\item [(a)] For every positive $\lambda_0 \in \Lambda$, the set
		$
		\{t \in \Lambda \colon 0 \leq 2t\leq \lambda_0 \}
		$
		has a maximum.
		\item [(b)] Every $\Lambda$-metric space that satisfies axioms (1) and (2) also satisfies (3).
	\end{itemize}
\end{theorem} 

For $\Lambda = \zz$ and $\Lambda = \rr$, we recover the statement of \cite[Lemma I.1.1]{ms84}, as every bounded subset of $\zz$ and $\rr$ has a supremum. Condition (a) can also be satisfied if every element in $\Lambda$ is divisible by $2$, showing that axiom (3) follows from (1) and (2) for groups like $\qq,\, \zz[\frac{1}{2}]$ or $\rr \times \rr$ with the lexicographical ordering. On the other hand, Theorem \ref{thm:main} also shows that in general, (for example for $\Lambda = \mathbb{Z}[\frac{1}{3}]$ or $\Lambda = \mathbb{Z}[\sqrt{2}]$), axiom (3) is independent of (1) and (2).

When $\Lambda$-trees were introduced in \cite{ms84}, $X$ was required to be uniquely geodesic as part of the definition of $\Lambda$-trees. The uniqueness however already follows from the conditions (1) and (2), see \cite[Lemma I.3.6]{ch01}. We show in Proposition \ref{prop:1_3_to_unique_1} that uniqueness also follows from the conditions (1) and (3).
\begin{prop}\label{prop:1_3_to_unique_1}
	Let $(X,d)$ be a geodesic $\Lambda$-metric space. If $(X,d)$ satisfies axiom (3), then $(X,d)$ is uniquely geodesic.
\end{prop}

We show further that for many ordered abelian groups $\Lambda$, axiom (2) does not follow from axioms (1) and (3).

\begin{theorem}\label{thm:sec}
	Let $\Lambda \neq 2 \Lambda$ or let $\Lambda$ be the additive group of an ordered field. Then axiom (2) is independent of the axioms (1) and (3).
\end{theorem}

Theorem \ref{thm:sec}, does not cover all the cases: the group $\Lambda_2 := \zz[\frac{1}{2}]$ for instance does not satisfy the conditions of Theorem \ref{thm:sec}. It is not even clear whether every uniquely geodesic $\Lambda_2$-metric space is a $\Lambda_2$-tree. We would therefore like to propose the following two questions. In view of Proposition \ref{prop:1_3_to_unique_1}, a positive answer to the first question would imply a positive answer to the second.

\begin{question}\label{q:2_independent}
	Let $\Lambda$ be a non-trivial ordered abelian group. Is axiom (2) independent of the axioms (1) and (3)?
\end{question}

\begin{question}\label{q:unique_independent}
	Let $\Lambda$ be a non-trivial ordered abelian group. Is there a uniquely geodesic $\Lambda$-metric space that does not satisfy axiom (2)?
\end{question}

In applications, the algebraic condition $\Lambda = 2 \Lambda$ is often satisfied. An important source of ordered abelian groups is as the image of a valuation $v \colon \ff_{>0} \to \Lambda$ of an ordered field $\ff$. If $\ff$ is \emph{Euclidean}, meaning that all positive elements have square roots in the field, then $\Lambda=2\Lambda$, because $v$ turns multiplication into addition. In particular, all real closed fields are Euclidean. For example, Brumfiel \cite{br88} starts with a non-Archimedean Euclidean field $\ff$ with a valuation $v \colon \ff_{>0} \to{ \Lambda \subseteq \rr}$ and obtains a $\Lambda$-metric space $X$ as a quotient of a non-standard hyperbolic plane over $\ff$. By Theorem \ref{thm:main}, it suffices to check axioms (1) and (2) when proving that the obtained $\Lambda$-metric space is a $\Lambda$-tree. 

{\bf Formaliziation in the proof assistant \lean.} 
The \lean\ theorem prover is a proof assistant developped mainly by Leonardo de Moura at Microsoft Research \cite{mkadr15}. There is an extensive community-built mathematical library \texttt{mathlib} \cite{mC19}, which by now contains a large part of undergraduate mathematics and also some more advanced projects such as a formalization of perfectoid spaces \cite{bcm20}. We build on the definition of ordered abelian groups already present in \texttt{mathlib} to formalize parts of this paper. We first formalize the notions of $\Lambda$-metric spaces and $\Lambda$-trees. We show some basic properties such as reparametrizations of segment maps. We give complete formal proofs for Lemma \ref{lem:1_2_to_unique_1}, Lemma \ref{lem:lem} and the direction (a) implies (b) of Theorem \ref{thm:main}. The \lean-files can be found \cite{app23}. We believe that the formalization of proofs can considerably improve the reliability of new mathematical results.

\newpage

{\bf Relation to $\Lambda$-buildings.} The notion of a $\Lambda$-tree was introduced in \cite{ms84} to generalize the notion of $\rr$-trees which itself is a generalization of trees (or $\zz$-trees) from graph theory. Analogously Bennett \cite{ben90} gave a system of axioms (A1) - (A6) for affine $\Lambda$-buildings, which generalize affine $\rr$-buildings which themselves generalize the classical affine buildings (or $\zz$-buildings). Indeed, the rank $1$ affine $\Lambda$-buildings are exactly the $\Lambda$-trees without leaves. While axiom (3) ensures that the tree-segments intersect where they should, axiom (A6) has a similar role in the theory of affine $\Lambda$-buildings. Just like (3) follows from (1) and (2) in the case $\Lambda=\rr$, (A6) follows from the other axioms for affine $\rr$-buildings, see \cite{par00}. To show that (A6) is independent of (A1) - (A5), Bennett (\cite{ben94}, Remark 3.3) claims to give an example of a $\qq$-metric space that satisfies all axioms of a rank $1$ affine $\qq$-building except for (A6). However, there is a mistake in the example: not all distances are $\qq$-valued. Indeed, Theorem 1 implies, that no rank $1$ counterexample for $\Lambda = \qq$ can exist. However, the example in \cite{ben94} can be modified by taking $\Lambda = \zz[\frac{1}{3}]$ instead of $\qq$ to show that axiom (A6) is independent of the other axioms. The axioms of affine $\Lambda$-buildings have been studied and simplified in \cite{bss}, and it would be interesting to know under which assumptions on $\Lambda$ (A6) follows from the other axioms, thus simplifying the axioms further and possibly generalizing Theorem \ref{thm:main}.

{\bf Structure of the paper.} In Section \ref{sec:def} we state the definitions and some elementary results in the theory of $\Lambda$-trees, such as Proposition \ref{prop:1_3_to_unique_1}. In Section \ref{sec:thm} we prove the implication (a) $\implies$ (b) of Theorem \ref{thm:main}. In Section \ref{sec:examples} we construct counterexamples to show the implication $\neg$(a) $\implies$ $\neg$(b) of Theorem \ref{thm:main} and to show Theorem \ref{thm:sec}.

The author would like to thank L. De Rosa and X. Flamm for insightful discussions.

\section{Definitions and elementary results}\label{sec:def}

We define the notion of a $\Lambda$-tree following Chiswell's book \cite{ch01}, where more details can be found. Let $(\Lambda,+)$ be an ordered abelian group. A set $X$ together with  a $\Lambda$-valued function
$$
d \colon X \times X \to \Lambda 
$$
that is positive definite, symmetric and satisfies the triangle inequality is called a \emph{$\Lambda$-metric space}.
A \emph{closed $\Lambda$-interval} is a set of the form
$$
[a,b]:=  \{\lambda \in \Lambda \colon a \leq \lambda \leq b\}
$$
for $a\leq b\in \Lambda$. Closed $\Lambda$-intervals with $d_\Lambda(t,t')=|t'-t|$ for $t,t'\in [a,b]$ are $\Lambda$-metric spaces. Let $X$ be any $\Lambda$-metric space. An isometric embedding $\varphi \colon [a,b] \to X$ is called a \emph{parametrization} of its image $s=\varphi([a,b])\subseteq X$, which is called a \emph{segment}. Note that the set of \emph{endpoints} $\{\varphi(a), \varphi(b)\}\subseteq X$ is independent of the parametrization $\varphi$ of the segment $s$. A $\Lambda$-metric space is \emph{geodesic} if for any two points $p,q\in X$ there exists a segment that has $p$ and $q$ as endpoints. If there is only one such segment, then $X$ is called \emph{uniquely geodesic} (or \emph{geodesically linear} in \cite{ch01}). In a uniquely geodesic $\Lambda$-metric space, we denote a segment $s$ with endpoints $x,y \in X$ by $s=[x,y]$. 

A $\Lambda$-metric space $(X,d)$ is a \emph{$\Lambda$-tree} if it satisfies the following three axioms, two of which are illustrated in Figure \ref{fig:axioms_2_3}.
\begin{itemize}
	\item[(1)] $(X,d)$ is geodesic.
	\item[(2)] If two segments $s,s'\subseteq X$ intersect in a single point $s\cap s' = \{p\}$, which is an endpoint of $s$ and $s'$, then their union $s\cup s'$ is a segment.
	\item[(3)] If two segments $s,s'\subseteq X$ have an endpoint in common, then their intersection $s\cap s'$ is a segment. 
\end{itemize}

\begin{figure}[t]
	\centering
	\includegraphics[scale=0.7]{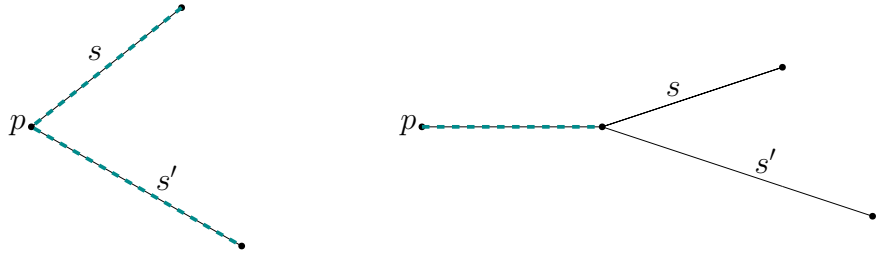}
	\caption{Illustration of axiom (2) (left) and axiom (3) (right) of $\Lambda$-trees. In axiom (2) the segments $s, s'$ are only allowed to intersect in one point.}
	\label{fig:axioms_2_3}
\end{figure}

In the presence of either axiom (2) or axiom (3), geodesic $\Lambda$-metric spaces are uniquely geodesic. For axiom (2), this is Lemma \ref{lem:1_2_to_unique_1} and was shown in \cite[Lemma I.3.6]{ch01}. For axiom (3), this is Proposition \ref{prop:1_3_to_unique_1} and we give a proof here.
 
\begin{lemma}[Lemma I.3.6 in \cite{ch01}]\label{lem:1_2_to_unique_1}
	Let $(X,d)$ be a geodesic $\Lambda$-metric space. If $(X,d)$ satisfies axiom (2), then $(X,d)$ is uniquely geodesic.
\end{lemma}

\begin{repprop}{prop:1_3_to_unique_1}
	Let $(X,d)$ be a geodesic $\Lambda$-metric space. If $(X,d)$ satisfies axiom (3), then $(X,d)$ is uniquely geodesic.
\end{repprop}
\begin{proof}
	Let us assume that axiom (3) is satisfied and let $s, s'$ be two segments with common endpoints $p$ and $q$. By axiom (3), their intersection $s\cap s'$ is a segment. We note that $p,q \in s \cap s'$. Let $\varphi \colon [0,d(p,q)] \to X$ be a parametriztion of (potentially a subsegment of) $s \cap s'$ with $\varphi(0)=p$ and $\varphi(d(p,q))=q$. For $r \in s$, we have $d(p,r) \leq d(p,q)$ and hence $\varphi(d(p,r)) \in s\cap s' \subseteq s$. Since $r$ is the unique point on $s$ with distance $d(p,r)$ from $p$, we have $\varphi(d(p,r)) = r$ 
	showing $s\subseteq s\cap s'$. Similarly $s' \subseteq s\cap s'$ and hence $s=s\cap s' = s'$.
\end{proof}

 The idea for the following lemma appears also in Chiswell's proof of Lemma \ref{lem:1_2_to_unique_1}. We will use this lemma in the proof of Theorem \ref{thm:main}.

\begin{lemma} \label{lem:lem}
	Let $(X,d)$ be a uniquely geodesic $\Lambda$-metric space. Let $s=[x,y]$ be a segment, $z \in s$ and $p\in X$. Then
	$$
	[x,z] \cap [z,p] = \{z\} \quad \text{ or } \quad [y,z] \cap [z,p] = \{z\}.
	$$
\end{lemma}

\begin{figure}[h]
	\centering
	\includegraphics[scale=1]{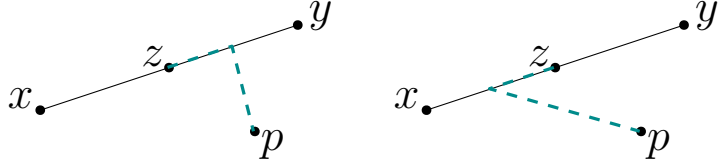}
	\caption{The lemma states that the segment $[z,p]$ cannot intersect both $[x,z]$ and $[y,z]$ outside of $\{z\}$. It is also possible that $[z,p]$ intersects $[x,y]$ in only one point $z$, in which case both possibilities in the lemma are correct.}
	\label{fig:lem}
\end{figure}

\begin{proof}
	For $a:={d(z,p) \in \Lambda}$, let $\varphi \colon [0,a] \to X$ be a parametrization of $[z,p]$ with $\varphi(0) = z$ and $\varphi(a)=p$. Note that $z$ is an endpoint of $[z,p]$. Let $x' \in [x,z]\cap [z,p] $ and $y' \in [y,z] \cap [z,p]$. We would like to show that $x'=z$ or $y'=z$. 
	
	There are $t_x,  t_y \in [0,a]$, such that $\varphi(t_x)=x'$ and $\varphi(t_y)=y'$, and thus $d(x',y') = |t_y-t_x|$. Looking at the segment $[x,y]$, we also know that $z\in [x',y'] \subseteq s$, since $[x',y'] \subseteq [x,y] = [x,z] \cup [z,y]$ is unique. So $d(x',y') = d(x',z)+d(z,y') = t_x + t_y$. The equation $|t_y-t_x| = t_y+t_x$ can only be resolved if $t_x=0$ or $t_y= 0$, so $x'=z$ or $y' = z$. This concludes the proof of the lemma.
\end{proof}

\section{Dependence results} \label{sec:thm} 

Recall that $r \in \Lambda$ is called the \emph{maximum of a subset $S\subseteq \Lambda$} if $r \in S$ and for all $s \in S$, $s \leq r$. We now state the main theorem and prove (a) $\implies$ (b). This proof of this direction is formalized in \lean \ in \cite{app23}. The converse direction is proved in Section \ref{sec:examples}.

\begin{reptheorem}{thm:main}
	Let $\Lambda$ be an ordered abelian group. The following are equivalent:
	\begin{itemize}
		\item [(a)] For every positive $\lambda_0 \in \Lambda$, the set
		$
		\{t \in \Lambda \colon 0 \leq 2t \leq \lambda_0 \}
		$
		has a maximum.
		\item [(b)] Every $\Lambda$-metric space that satisfies axioms (1) and (2) also satisfies (3).
	\end{itemize}
\end{reptheorem} 

\begin{proof}
	
We will show (a) implies (b). So let $\Lambda$ be an ordered abelian group such that for every positive $\lambda_0 \in \Lambda$ the set $[0,\lambda_0/2] := \{\lambda \in \Lambda \colon 0 \leq 2\lambda \leq \lambda_0 \}$ has a maximum and let $(X,d)$ be a $\Lambda$-metric space which is (1) geodesic and (2) whenever two segments intersect in a single common endpoint, then their union is a segment. Let $s, s'$ be two segments with a common endpoint $x\in X$. We want to prove that (3) the intersection $s\cap s'$ is a segment. 

By Lemma \ref{lem:1_2_to_unique_1}, $X$ is uniquely geodesic. Let $y$ be the other endpoint of $s$ and let $z$ be the other endpoint of $s'$, i.e. $s=[x,y]$ and $s' = [x,z]$. For $a:=d(x,y)$ and $b:=d(x,z)$, we have parametrizations
\begin{align*}
    \varphi \colon [0,a] &\to X, \quad\quad \varphi(0) = x , \quad \varphi(a)=y \\
    \psi \colon [0,b] &\to X, \quad\quad  \psi(0) = x, \quad \psi(b) = z
\end{align*}
of the segments $s$ and $s'$. Without loss of generality, let $a\leq b$. We consider $\tilde{z}=\psi (a) \in [x,z]$. By (1) there is a segment $[y, \tilde{z}]$ with parametrization
\begin{align*}
	\sigma \colon [0,d(y,\tilde{z})] &\to X,  \quad \quad \sigma(0) = y , \quad \sigma(d(y,\tilde{z}))=\tilde{z}.
\end{align*}
We consider the subset $S = \{\lambda \in \Lambda \colon 0 \leq 2\lambda \leq d(y,\tilde{z})\}$ which by assumption (a) has a maximum $r\in \Lambda$. We define $\ell = d(y,\tilde{z}) - r$ and note that $ \ell \geq r$, since $2r \leq d(y,\tilde{z})$. Next, we show that $\ell \leq a$ : if $d(y,\tilde{z}) \leq a$, then $\ell \leq d(y,\tilde{z}) \leq a$ is clear, otherwise we have $d(y,\tilde{z})-a \geq 0$, and by the triangle inequality
$$
2(d(y,\tilde{z})-a) = d(y,\tilde{z}) + d(y,\tilde{z}) - 2a \leq 2a + d(y,\tilde{z}) -2a = d(y,\tilde{z})
$$
and hence $d(y,\tilde{z})-a \in S$, which implies $d(y,\tilde{z}) - a \leq r$, i.e. $\ell \leq a$. 

We can now define four points $p := \sigma(r) , q := \sigma(\ell), y':= \varphi(a-\ell)$ and $z' := \psi(a-\ell)$, such that $d(y,y') =d(y,q) = \ell = d(\tilde{z},p) = d(\tilde{z},z')$. The situation is illustrated in Figure \ref{fig:setting}. 

\begin{figure}[h]
	\centering
	\includegraphics[scale=0.9]{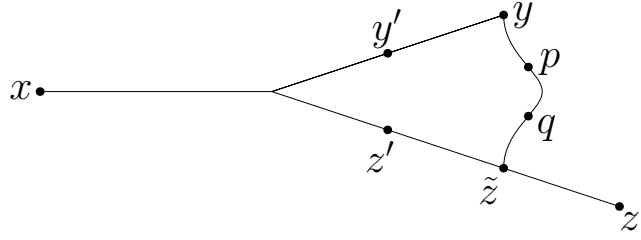}
	\caption{Given the segments $s= [x,y], s' = [x,z] \subseteq X$, we can construct points $\tilde{z}, p, q, y'$ and $z'$ with $d(y,y')=d(y,q)=d(\tilde{z},z')=d(\tilde{z},p) = \ell$. The idea of the proof is to show that $p=y'=z'$, and $s \cap s' =[x,p]$ and thus a segment. }
	\label{fig:setting}
\end{figure}

Our goal is to show that $y' = z'$ and to then conclude that $s\cap s' = [x,y']$ is a segment. We apply Lemma \ref{lem:lem} to the segments $s=[x,y]$ (and $\tilde{s}'=[x,\tilde{z}]$), $y' \in s$ (and $z' \in \tilde{s}'$) and $p$ to create the following case distinction, illustrated in Figure \ref{fig:cases}. In all cases we prove that $y'=z'$.

\begin{figure}[h]
\centering
\includegraphics[scale=1.1]{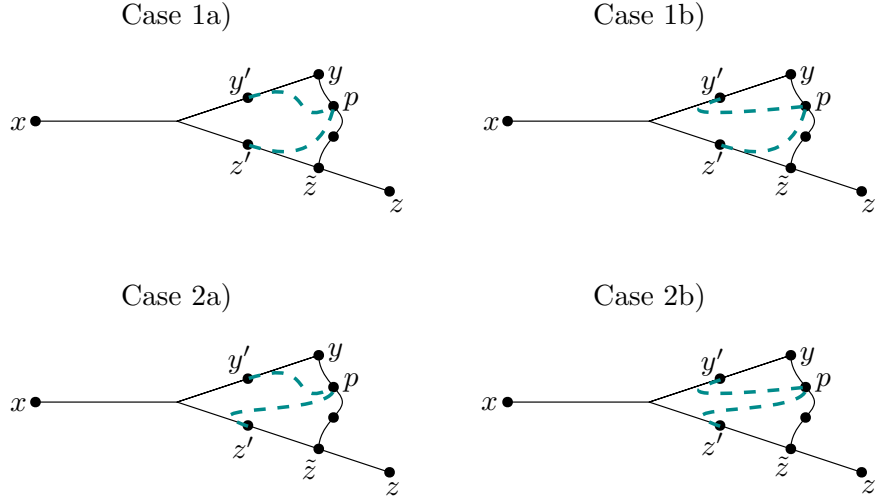}
\caption{We apply Lemma \ref{lem:lem} to two segments. The two possibilities in the lemma lead to four cases we have to consider.}
\label{fig:cases}
\end{figure}

\textbf{Case 1a)}: $[x,y']\cap[y',p] = \{y'\}$ and $[x,z']\cap [z',p] = \{z'\}$.

In this case we have by axiom (2) that $[x,y'] \cup [y',p]$ and $[x,z']\cup[z',p]$ are segments from $x$ to $p$. Since the segment $[x,p]$ is unique, and $d(x,y')=a-\ell=d(x,z')$, it follows that $y'=z'$. 

\textbf{Case 1b)}: $[y,y']\cap[y',p] = \{y'\}$ and $[x,z']\cap [z',p] = \{z'\}$.

In this case, we have by axiom (2) that $[y,y']\cup [y',p] = [y,p]$ is a segment. Since $r \leq \ell = d(y,y') \leq d(y,y') + d(y',p) = d(y,p) = r$, we have in this case $\ell = r$ and $d(y',p)=0$, hence $y'=p$. Now we apply (2) again to see that $[x,z']\cup [z',p] = [x,p] = [x,y']$ is a segment. By uniqueness it is a subsegment of $s=[x,y]$, and since $d(x,z') = a-\ell = d(x,y')$ we have that $z' = y' = p$. 

\textbf{Case 2a)}: $[x,y']\cap[y',p] = \{y'\}$ and $[\tilde{z},z']\cap [z',p] = \{z'\}$.

This case is similar to 1b). We have by axiom (2) that $[\tilde{z},z']\cup [z',p] = [\tilde{z},p]$ is a segment. Since $d(\tilde{z},z') \leq d(\tilde{z},z') + d(z',p) = d(\tilde{z},p) =\ell = d(\tilde{z},z')$, we have $d(z',p)=0$ and hence $z'=p$. Now we apply (2) again to see that $[x,y']\cup [y',p] = [x,p] = [x,z']$ is a segment. By uniqueness it is a subsegment of $s=[x,\tilde{z}]$, and since $d(x,y') = a-\ell = d(x,z')$ we have that $y' = z' = p$. 

\textbf{Case 2b)}: $[y,y']\cap[y',p] = \{y'\}$ and $[\tilde{z},z']\cap [z',p] = \{z'\}$.

In this case $y'=p$ follows as in the first part of Case 1b) and $z'=p$ follows as in the first part of Case 2a). 

In all cases we have established that $y' = z'$. By uniqueness, we see that $[x,y']=[x,z'] \subseteq s\cap s'$. It remains to show that $s\cap s' \subseteq [x,y'] = [x,z']$. Let $t,t' \in \Lambda$ such that $\varphi(t)=\psi(t') \in s\cap s'$. By uniqueness of $[x,\varphi(t)]=[x,\psi(t')]$, $t=t'$ has to hold. We use the triangle inequality to get $d(y,\tilde{z}) \leq d(y,\varphi(t))+d(\psi(t),\tilde{z}) = 2(a-t)$, hence
$$
2(d(y,\tilde{z})-(a-t)) \leq 2(a-t) + d(y,\tilde{z})  -2(a-t) = d(y,\tilde{z})
$$
from which we conclude that $d(y,\tilde{z}) -(a-t) \in S = \{\lambda \in \Lambda \colon 0 \leq 2\lambda \leq d(y,\tilde{z})\}$ (or is smaller than $0$), hence $d(y,\tilde{z})-(a-t) \leq r$, so
$\ell = d(y,\tilde{z})-r \leq a-t$ and $t \leq a-\ell = d(y,y') = d(\tilde{z},z')$. Thus $\varphi(t) \in [x,y'] = [x,z']$ and $s\cap s' = [x,y']$, proving axiom (3).
\end{proof}

\section{Independence results}\label{sec:examples}
In this section we construct three counterexamples $X_1, X_2, X_3$. Their properties are summarized in the following proposition.
\begin{prop} {\color{white}.}
	\begin{itemize}
		\item[(i)] When condition (a) of Theorem \ref{thm:main} does not hold, there is a $\Lambda$-metric space $X_1$ satisfying axioms (1) and (2), but not axiom (3).
		\item[(ii)] When $\Lambda$ is the additive group of an ordered field, there is a $\Lambda$-metric space $X_2$ satisfying axioms (1) and (3), but not axiom (2).
		\item[(iii)] When $\Lambda \neq 2\Lambda$, there is a $\Lambda$-metric space $X_3$ satisfying axioms (1) and (3), but not axiom (2).
	\end{itemize}
\end{prop}

The existence of $X_1$ proves that $\neg$(a) implies $\neg$(b) in Theorem \ref{thm:main}. Together with the results of Section \ref{sec:thm} this concludes the proof of Theorem \ref{thm:main}. The existence of $X_2$ and $X_3$ together imply Theorem \ref{thm:sec}.

\subsection{Construction of $X_1$}

We will construct a $\Lambda$-metric space $X_1$ that satisfies axioms (1) and (2), but not (3), whenever the condition
\begin{itemize}
	\item [(a)] For every positive $\lambda_0 \in \Lambda$, the set
	$
	\{t \in \Lambda \colon 0 \leq 2t \leq \lambda_0 \}
	$
	has a maximum.
\end{itemize}
is not satisfied. So let $\lambda_0 \in \Lambda$ be an element such that the set 
$$
I=\{t \in \Lambda \colon 0 \leq 2t \leq \lambda_0 \}
$$
has no maximum and is not empty. Intuitively, $I$ is an interval which has a minimum $0$, but no maximum. In particular, $\lambda_0$ is not divisible by $2$. We construct a $\Lambda$-metric space $X_1$ as illustrated in Figure \ref{fig:counterexample}. We will sometimes use the fact that $\Lambda$ and $[0,\lambda_0]$ are $\Lambda$-trees without proof.

\begin{figure}[h]
	\centering
	\includegraphics[scale=1]{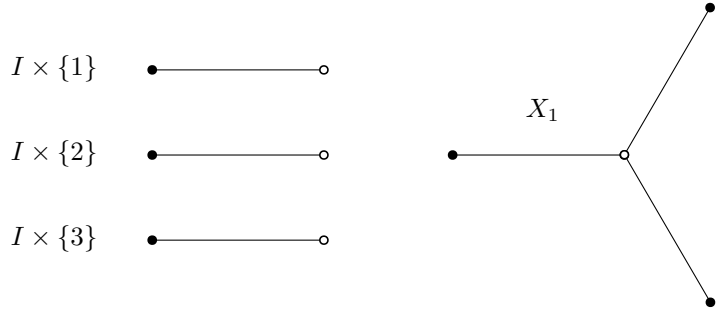}
	\caption{When $\Lambda$ does not satisfy condition (a), we can construct the $\Lambda$-metric space $X_1$ which satisfies (1) and (2), but not (3). The idea is that the branch-point is missing.}
	\label{fig:counterexample}
\end{figure}

\begin{lemma} \label{lem:metric}
The set
$$
X_1 = \bigcup_{i=1}^3 I \times \{i\}
$$
with the distance function
\begin{align*}
	d \colon X_1 \times X_1 & \to \Lambda \\
	((x,i),(y,j)) & \mapsto \left\{\begin{array}{ll}
		|x-y| & \text{if } i = j \\
		\lambda_0 - x - y & \text{else}  \\
	\end{array}\right.
\end{align*}
is a $\Lambda$-metric space. 
\end{lemma}
\begin{proof}
	Positive definiteness and symmetry follow directly, the triangle inequality requires a case distinction: If all three points are in $I\times \{i\}$, then the triangle inequality follows from the one on $\Lambda$. For $(x,i), (y,i), (z,j) \in X_1$ with $i\neq j$, where without loss of generality $x \leq y$, an explicit calculation shows
	\begin{equation*} \label{eq:triangle_eq}
		d((x,i),(y,i)) + d((y,i),(z,j)) = d((x,i),(z,j)).
	\end{equation*}
	For $(x,i), (y,j), (z,k) \in X_1$ with distinct $i,j,k$ a calculation shows that the strong triangle inequality
	\begin{equation*}\label{eq:triangle_strong}
		d((x,i),(z,k)) < d((x,i),(y,j)) + d((y,j),(z,k))
	\end{equation*}
	holds.
\end{proof}
\begin{lemma}\label{lem:axiom1}
	Let $(x,i),(y,j) \in X_1$ and $x\leq y$ or $i\neq j$. The image of the map
	\begin{align*}
	\varphi \colon 	[0,D] & \to X_1 \\
		t & \mapsto \left\{\begin{array}{ll}
			(x+t,i) & \text{if} \quad 2(x+t) < \lambda_0 \\
			(\lambda_0-x-t , j)  & \text{if}\quad 2(x+t) > \lambda_0 ,\\
		\end{array}\right. 
	\end{align*}
	where $D=d((x,i),(y,j))$, is a segment from $(x,i)$ to $(y,j)$. The $\Lambda$-metric space $(X_1,d)$ is geodesic. 
\end{lemma}
\begin{proof}
	If $i=j$ and $x\leq y$, $D=y-x$ and $2(x+t)<\lambda_0$ is satisfied for all $t\in [0,D]$. It is then clear that the map $\varphi$ is an isometry with $\varphi(0) = (x,i)$ and $\varphi(D)=(y,j)$. The corresponding segment is given by $[x,y]\times\{i\}$. 
	
	For $i\neq j$, we have $D=\lambda_0-x-y$. The map $\varphi$ is well defined since if $2(x+t)>\lambda_0$, then $2(\lambda_0-x-t) = 2\lambda_0 - 2(x+t)  < 2\lambda_0 - \lambda_0 = \lambda_0$.	The endpoints behave correctly, $\varphi(0) = (x,i)$ and $\varphi(D) = (y,j)$ since $2(x+D) = 2(x+\lambda_0-x-y) = \lambda_0 + (\lambda_0-2y) > \lambda_0$ and hence $\varphi(D) = (\lambda_0 - x - D , j) = (y,j)$. It remains to show that $\varphi$ is an isometry. Let $t,t' \in [0,D]$ with $t<t'$. If $2(x+t)$ and $2(x+t')$ are both either larger or smaller than $\lambda_0$, then $\varphi(t)$ and $\varphi(t')$ are contained in the same copy of $I$ and the isometries follow from calculations in $\Lambda$. If $2(x+t)<\lambda_0$ and $2(x+t')>\lambda_0$, then
	$$
	d(\varphi(t), \varphi(t')) = \lambda_0 - (x+t) - (\lambda_0-x-t') = t'-t.
	$$ 
	Given any two points $(x,i)$ and $(y,j)$ in $X$ we may assume without loss of generality $x\leq y$. The image of the map $\varphi$ thus gives us the segment.
\end{proof}
From Lemma \ref{lem:axiom1} we can conclude that $[0,\lambda_0]$ is isometric to $I\times\{i\} \cup I \times \{j\}$ whenever $i\neq j$, (choose $x=y=0$). The segments described in Lemma \ref{lem:axiom1} can therefore be viewed as subsets of $[0,\lambda_0]$. The next lemma states that the segments in Lemma \ref{lem:axiom1} are the only ones.
\begin{lemma}\label{lem:unique_axiom1}
	The $\Lambda$-metric space $(X_1,d)$ is uniquely geodesic.
\end{lemma}
\begin{proof}
	Every segment from $(x,i)$ to $(y,j)$ admits a parametrization of the form $\psi \colon [0,D] \to X_1$, where $\psi(0) = (x,i)$ and $\psi(D) = (y,j)$. We first claim that for every $t \in [0,D]$, $\psi(t) \in I\times\{i\} \cup I \times \{j\}$. If this were not the case, then we would have
	\begin{align*}
		d((x,i),\psi(t)) &= \lambda_0 - x - r \\
		d((y,j),\psi(t)) &= \lambda_0 - y - r 
	\end{align*}
    for some $r\in I$ with $(r,k) = \psi(t)$. We would then have
    $$
    D = d((x,i),(y,j)) =  d((x,i),\psi(t)) + d((y,j),\psi(t)) = 2\lambda_0 - x - y - 2r
    $$
    since $\psi$ is an isometry. If $i=j$ and without loss of generality $x<y$ then
    $$
    y-x = 2\lambda_0 - x - y - 2r
    $$
    which contradicts $\lambda_0>2r$ and $\lambda_0 > 2y$. If $i\neq j$, then
    $$
    D = \lambda_0 - x - y = 2\lambda_0 - x - y - 2r
    $$
    which contradicts $\lambda_0 > 2r$.
    
    Let $k\neq i$ and $k=j$ if $j\neq i$. We can use Lemma \ref{lem:axiom1} with $(0,i)$ and $(0,k)$ to get an isometry $f \colon  [0,\lambda_0]  \to I\times\{i\} \cup I \times\{k\}$. 
    Since segments in $[0,\lambda_0]$ are unique, also segments in $f([0,\lambda_0]) = I\times\{i\}\cup I \times \{k\}$, such as $\psi([0,D])$ are unique, concluding the proof.
\end{proof}

\begin{lemma}
	The $\Lambda$-metric space $(X_1,d)$ satisfies axiom (2).
\end{lemma}
\begin{proof}
	Let $s, s'$ be two segments that intersect exactly in one common endpoint $(x,i) \in X_1$. We will show that there is a $j\neq i$ such that both $s$ and $s'$ are contained in $I\times\{i\}\cup I\times\{j\}$. If the other endpoint of $s$ is in $I\times\{j\}$ and the other endpoint of $s'$ is in $I\times\{k\}$ for distinct $i, j,k$, then $s \subset I\times\{i\}\cup I \times \{j\}$ and $s' \subset I\times\{i\}\cup I \times \{k\}$ and we know how the segments look like by the description of segments in Lemmas \ref{lem:axiom1} and \ref{lem:unique_axiom1}. Let $y\in I$ with $x<y$, which exists because $x$ is not a maximum of $I$, (we are assuming $I$ does not have a maximum). But then $(y,i)\in s\cap s'$ by the description of the segments in Lemma \ref{lem:axiom1}, which contradicts the assumption $s\cap s' = \{(x,i)\}$. 
	
	We conclude that there is a $j\neq i$ such that both $s$ and $s'$ are contained in $I\times\{i\}\cup I\times\{j\}$, which is isometric to $[0,\lambda_0]$.
	Since axiom (2) holds in $[0,\lambda_0]$ also $s\cup s'$ is a segment.
\end{proof}

\begin{lemma}
	The $\Lambda$-metric space $(X_1,d)$ does not satisfy axiom (3).
\end{lemma}
\begin{proof}
	Let $s$ be the unique segment from $(0,1)$ to $(0,2)$ and let $s'$ be the unique segment from $(0,1)$ to $(0,3)$. The point $(0,1)$ is a common endpoint of $s$ and $s'$. Since $s\cap s' = I\times \{1\}$, and $I$ does not have a maximum, $s\cap s'$ is not isometric to a closed $\Lambda$-interval and hence is not a segment.  
\end{proof}

We have constructed a $\Lambda$-metric space $(X_1,d)$ which satisfies axioms (1) and (2) but not (3). This shows $\neg$(a) $\implies \neg$(b) in Theorem \ref{thm:main} and thus completes its proof.

This example shows that axiom (3) forces the branchpoints to be part of the $\Lambda$-tree. For groups such as $\Lambda=\qq$ (or $\zz[\frac{1}{2}]$), one might be tempted to construct $\Lambda$-metric spaces with branchpoints at an irrational distance (or at $\frac{1}{3}$). However, that construction fails, since the resulting metric does not take values in $\Lambda$. The number $\frac{1}{2}$ plays a special role, as reflected in condition (a) of Theorem \ref{thm:main}.

\subsection{Construction of $X_2$} It was noted in \cite{ab87}, that for $\Lambda = \mathbb{R}$, axiom (2) is independent from the axioms (1) and (3), consider for instance $\rr^2$ with the Euclidean distance. In view of Proposition \ref{prop:1_3_to_unique_1}, any such example has to be geodesically unique. The idea of this section is thus to build a uniquely geodesic $\Lambda$-metric space $X_2$ which is $2$-dimensional in some sense.

Let $\Lambda$ be the additive group of a ordered field. The idea is to create a version of the $\ell_1$-distance (also called the Manhattan-distance) but using polar coordinates instead of the usual Cartesian coordinates. The distance can be thought as the length of the shortest path consisting only of radial and circumferential movements. To be uniquely geodesic, we restrict ourselves to a subset. The situation is illustrated in Figure \ref{fig:spider}.
We consider the set
$$
X_2 = (0,2]\times [0,1] = \{ (t, \varphi) \in \Lambda^2 \colon 0 < t \leq 2 ,\ 0 \leq \varphi \leq 1 \}
$$
and endow it with the function $d \colon X_2 \times X_2 \to \Lambda$ given by
\begin{align*}
	d((t_1,\varphi_1), (t_2,\varphi_2)) =  |t_2 - t_1| + \min \{t_1,t_2\} \cdot |\varphi_2 - \varphi_1| .
\end{align*}

\begin{figure}[h]
	\centering
	\includegraphics[scale=1.1]{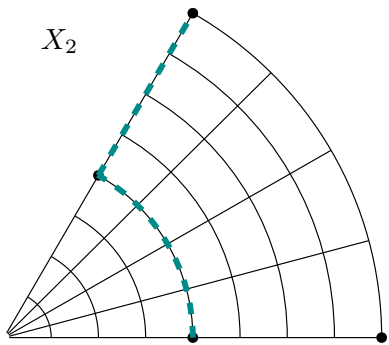}
	\caption{In the $\Lambda$-metric space $X_2$ segments consist of a radial part and the shorter of two circumferential parts. Unlike in the Manhattan-distance, segments are unique. }
	\label{fig:spider}
\end{figure}

\begin{lemma}
	The pair $(X_2, d)$ is a $\Lambda$-metric space.
\end{lemma}
\begin{proof}
	Symmetry is clear from the definition and positive definiteness follows from the fact that $\Lambda$ is a field. We prove the triangle inequality. Let $p_i = (t_i, \varphi_i) \in X_2$ for $i \in \{1,2,3\}$ be three points. To prove $d(p_1,p_2) \leq d(p_1, p_3) + d(p_3,p_2)$, we may assume without loss of generality that $t_1 \leq t_2$. We first show the triangle inequality for the case $t_3 \leq t_1$,
	\begin{align*}
		d(p_1,p_2) &= t_2-t_1 + t_1 |\varphi_2 - \varphi_1| \\
		&= t_2 -t_1 + (t_1 - t_3) |\varphi_2 - \varphi_1| + t_3 |\varphi_2 - \varphi_1| \\
		&< t_2 -t_1 + 2(t_1 - t_3) + t_3 |\varphi_2 - \varphi_1| \\
		&\leq  t_2 -t_3 + t_1 - t_3 + t_3 |\varphi_3 - \varphi_1| + t_3 |\varphi_2 - \varphi_3| \\
		&= d(p_1,p_3) + d(p_3, p_2)
	\end{align*}
	and then in the case $t_1 \leq t_3$
	\begin{align*}
		d(p_1, p_2) &= t_2-t_1 + t_1 |\varphi_2 - \varphi_1| \\ 
		&\leq t_2 - t_3 + t_3 - t_1 + t_1 |\varphi_3 - \varphi_1| + t_1 |\varphi_2 - \varphi_3| \\
		&\leq |t_3-t_2| + |t_3-t_1| + t_1 |\varphi_3 - \varphi_1| + t_1 |\varphi_2 - \varphi_3|\\
		&= d(p_1,p_2) + d(p_2, p_3).
	\end{align*}
\end{proof}
\begin{remark}
	We note that we could also define this metric on some other subsets of $\Lambda \times \Lambda$. For some examples like $(0,1]\times[0,2]$ and $\Lambda_{\geq 0} \times [0,2]$ we also get a $\Lambda$-metric space, but the triangle inequality would not hold on $(0,1]\times[0,3]$. We could also add one more point $(0,0)$ to $X_2$ and get a metric space that includes the "centerpoint". 
\end{remark}

\begin{lemma}\label{lem:2ax1}
	The $\Lambda$-metric space $(X_2,d)$ is uniquely geodesic. 
\end{lemma}
\begin{proof}
	For the points $p_i=(t_i,\varphi_i)$ with $i \in \{1,2\}$ we may assume without loss of generality that $t_1 \leq t_2$. We write $\varphi_0 := t_1\cdot |\varphi_2-\varphi_1|$, so that $d(p_1,p_2) = \varphi_0 + (t_2-t_1)$. One can check that 
	\begin{align*}
		f \colon [0,d(p_1,p_2)] & \to X_2 \\
		t & \mapsto  \left\{ \begin{array}{ll}
			\left(t_1, \frac{1}{\varphi_0} \left( t\cdot \varphi_2 + (\varphi_0 - t) \cdot \varphi_1 \right)\right) & \text{if} \quad t < \varphi_0 \\
			(t - \varphi_0 + t_1 , \varphi_2 ) &  \text{if} \quad t\geq \varphi_0
		\end{array} \right.
	\end{align*}
	defines a segment-map whose image $s$ is a segment with $p_1$ and $p_2$ as endpoints. Let $p=(t,\varphi) \in X \setminus s$. We claim that then $ d(p_1 , p_2) < d(p_1, p) + d(p,p_2)$. 
	If $t< t_1 \leq t_2$, then 
	\begin{align*}
		d(p_1,p_2) &= t_2-t_1 + t_1 |\varphi_2 - \varphi_1| \\
		&= t_2 -t_1 + (t_1 - t) |\varphi_2 - \varphi_1| + t |\varphi_2 - \varphi_1| \\
		&< t_2 -t_1 + 2(t_1 - t) + t |\varphi_2 - \varphi_1| \\
		&\leq  t_2 -t + t_1 - t + t |\varphi_3 - \varphi_1| + t |\varphi_2 - \varphi_3| \\
		&= d(p_1,p) + d(p, p_2).
	\end{align*}
	If $t_1 \leq t_2 < t$, then $t_2-t < t-t_2 = |t-t_2|$ and
	\begin{align*}
		d(p_1, p_2) &= t_2 - t + t - t_1 + t_1|\varphi_2 - \varphi_1|\\
		& < |t_2-t| + |t_1-t|  + t_1|\varphi_2 - \varphi| + t_1 |\varphi - \varphi_1| \\
		& \leq |t_2-t|   + t_2|\varphi_2 - \varphi| + |t_1-t| + t_1 |\varphi - \varphi_1| \\
		& = d(p_1,p) + d(p,p_2).
	\end{align*}
	If $t \in [t_1,t_2]$, we have
	\begin{align*}
		d(p_1, p_2) &= |t_2 - t_1| + t_1|\varphi_2 - \varphi_1|\\
		& \leq |t_2 - t| + |t - t_1| + t_1(|\varphi_2 - \varphi| + |\varphi - \varphi_1|) \\
		& = |t - t_1| + t_1 |\varphi - \varphi_1| + |t_2 - t| + t_1|\varphi_2 - \varphi| \\
		& \leq |t - t_1| + t_1 |\varphi - \varphi_1| + |t_2 - t| + t|\varphi_2 - \varphi| \\
		& = d(p_1,p) + d(p,p_2)
	\end{align*}
	and we see that in the last inequality, equality can only hold when $t_1|\varphi_2-\varphi|=t|\varphi_2-\varphi|$, i.e. when $t_1 = t$ or $\varphi_2 = \varphi$, which is exactly when $p\in s$. This shows the claim and we can conclude that $s$ is the unique segment from $p_1$ to $p_2$.
\end{proof}

\begin{remark}
	We note that the $\Lambda$-metric space $X_2$ could also be defined more generally for $\Lambda$ the additive group of an ordered integral domain, but the segment-parametrization in Lemma \ref{lem:2ax1} requires division.   
\end{remark}

\begin{lemma}
	The $\Lambda$-metric space $(X_2,d)$ satisfies axiom (3).
\end{lemma}
\begin{proof}
	We know from Lemma \ref{lem:2ax1} how segments look like. Let $s_1, s_2$ be two segments with a common endpoint $p$. Let $p_1=(t_1,\varphi_1)$ be the other endpoint of $s_1$ and $p_2=(t_2,\varphi_2)$ be the other endpoint of $s_2$. If $\varphi_1\neq\varphi_2$, then the intersection $s_1 \cap s_2$ is contained in $\{t\}\times [0,1]$ and we know that $s_1 \cap s_2$ is a segment, since $\{t\}\times [0,1]$ is a $\Lambda$-tree and satisfies axiom (3). If $\varphi_1 = \varphi_2$ then we know that $s_1 \cap s_2 \cap (0,1]\times\{\varphi_1\}$ is a segment of the $\Lambda$-tree $(0,1]\times\{\varphi_1\}$. By the description of the segments in Lemma \ref{lem:2ax1} $s_1 \cap s_2$ is a segment.
\end{proof}

\begin{lemma}
	The $\Lambda$-metric space $(X_2,d)$ does not satisfy axiom (2).
\end{lemma}
\begin{proof}
	We consider the segments $s_1 = [(1,0),(2,0)]$ and $s_2 = [(1,0),(2,1)]$. The two segments intersect exactly in one point $(1,0)$, but the union $s_1 \cup s_2$ is not a segment by the description of segments in Lemma \ref{lem:2ax1}.
\end{proof}

\subsection{Construction of $X_3$}

We now give an example for the case $\Lambda \neq 2\Lambda$. Let $a\in \Lambda$ not divisible by $2$. Consider the space $X_3=[0,3a]/\!\!\sim$, where we identify $0\sim 3a$. The $\Lambda$-metric is given by
$$
d(p,q) = \min\left\{ \left|q-p+k\cdot 3a \right| \colon k \in \zz\right\}, \quad \quad \text{for } p,q \in X_3.
$$
Intuitively we have turned the interval $[0,3a]$ into a circle with the shortest path metric. From the definition it follows that $(X_3,d)$ is a $\Lambda$-metric space and that it satisfies (1) and (3). Axiom (2) is however not satisfied, because the union $[0,a] \cup [a,2a]$ of two segments is not again a segment, since $d(0,2a) = a \neq 2a$. The actual segment connecting $0$ to $2a$ is $[2a, 3a]$.

\vspace{0cm}



\end{document}